\documentclass[11pt]{article}
\usepackage{amsmath}
\usepackage{amsthm}
\usepackage{amsfonts}
\usepackage{amssymb}
\usepackage{adjustbox}
\usepackage{authblk}
\usepackage{graphicx}
\usepackage{color}
\usepackage{rotating}
\usepackage{url}
\usepackage{algorithm}
\usepackage{algpseudocode}
\usepackage{epstopdf}
\usepackage{adjustbox}
\usepackage[normalem]{ulem}
\useunder{\uline}{\ul}{}
\usepackage[T1]{fontenc}
\usepackage[utf8]{inputenc}
\usepackage[font=small,labelfont=bf]{caption}
\usepackage{hyperref}

\newtheorem{theorem}{Theorem}

\newtheorem{lemma}{Lemma}

\newtheorem{observation}{Observation}

\DeclareMathOperator*{\convex}{conv}

\DeclareMathOperator*{\rr}{\mathbb{R}}

\newcommand{\socpr}{SOCr}
\newcommand{\socp}{SOCP}

\title{The convex hull of a quadratic constraint over a polytope}
\author[1]{Asteroide Santana\thanks{asteroide.santana@gatech.edu}}
\author[1]{Santanu S. Dey\thanks{santanu.dey@isye.gatech.edu}}
\affil[1]{\small School of Industrial and Systems Engineering, Georgia Institute of Technology}

\begin{document}
\maketitle
\begin{abstract}
A quadratically constrained quadratic program (QCQP) is an optimization problem in which the objective function is a quadratic function and the feasible region is defined by quadratic constraints. Solving non-convex QCQP to global optimality is a well-known NP-hard problem and a traditional approach is to use convex relaxations and branch-and-bound algorithms. This paper makes a contribution in this direction by showing that the exact convex hull of a general quadratic equation intersected with any bounded polyhedron is second-order cone representable. We present a simple constructive proof of this result.
\end{abstract}

\section{Introduction}\label{sec:intro}

A quadratically constrained quadratic program (QCQP) is an optimization problem in which the objective function is a quadratic function and the feasible region is defined by quadratic constraints. A variety of complex systems can be cast as an instance of a QCQP.  Combinatorial problems like MAXCUT~\cite{goemans1994879}, engineering problems such as signal processing~\cite{gharanjik2016iterative,KhabbazibasmenjVorobyov2014}, chemical process~\cite{haverly1978studies, meyer2006global, alfaki2013strong,dey2015analysis, gupte2017relaxations, tawarmalani2002convexification} and power engineering problems such as the optimal power flow \cite{BoseGayme2015,LiVittal2018,ChenAtamturkOren2016,kocuk2016strong} are just a few examples.

Solving non-convex QCQP to global optimality is a well-know NP-hard problem and a traditional approach is to use spacial branch-and-bound tree based algorithm. The computational success of any branch-and-bound tree based algorithm depends on the convexification scheme used at each node of the tree. Not surprisingly, there has been a lot of research on deriving strong convex relaxations for general-purpose QCQPs.  The most common relaxations found in the literature are based on Linear
programming (LP), second order cone programing (SOCP) or semi-definite programming
(SDP). Reformulation-linearization technique (RLT)~\cite{sherali2013reformulation,sherali1992new} is a LP-based hierarchy, Lasserre hierarchy or the sum-of-square hierarchy~\cite{lasserre2001global} is a SDP-based hierarchy which exactly solves QCQPs under some minor technical conditions and, recently, new  LP and SOCP-based alternatives to sum of squares optimization have also been proposed~\cite{ahmadi2014dsos}. While SDP relaxations are know to be strong, they don't always scale very well computationally. SOCP relaxations tend to be more computationally attractive, although they are often derived by further relaxing SDP relaxations \cite{Burer2014}.

Another direction of research focuses on convexification of functions, with the McCormick relaxation~\cite{mccormick1976computability} being perhaps the most classic example. In this case, a constraint of the form $f(x) = b$ is replaced with $\breve{f}(x) \leq b$ and $\hat{f}(x) \geq b$, where $\breve{f}$ is a convex lower approximation and $\hat{f}$ is a concave upper approximation of $f$. While there have been a lot of work in function convexification (see for instance \cite{al1983jointly,sherali1990explicit,androulakis1995alphabb,ryoo2001analysis,liberti2003convex,
benson2004concave,meyer2004trilinear,anstreicher2010computable,belotti2010valid,
bao2015global,misener2015dynamically,del2016polyhedral,sherali1997convex,Rikun1997,
meyer2005convex,tawarmalani2002convexification,tuy2016convex,locatelli2016polyhedral,
buchheim2017monomial,crama2017class,Adams2018,boland2017extended,speakman2017quantifying}) it is well-known that it does not necessarily yield the convex hull of the set $\{x \,|\, f(x) = b\}$. To the best of our knowledge, there have been much less work on explicit convexification of sets:~\cite{tawarmalani2013explicit,nguyen2013deriving,nguyen2011convexification,tawarmalani2010strong,
akshayguptethesis,kocuk2017matrix,rahman2017facets,davarnia2017simultaneous,LiVittal2018,Burer2017}. 

A related question when studying convex relaxations is that of representability of the exact convex hull of the feasible set: Is it LP, \socp \ or SDP representable? In \cite{DeySantana2018}, we prove that the convex hull of the so-called bipartite bilinear constraint (which is a special case of a quadratic constraint) intersected with a box constraint is \socp \ representable (\socpr). The proof yields a procedure to compute this convex hull exactly. Encouraging computational results are also reported in~\cite{DeySantana2018} in terms of obtaining dual bounds using this construction, which significantly outperform SDP and McCormick relaxations and also bounds produced by commercial solvers.

\section{Our result} \label{section: our result}
For an integer $t \ge 1$, we use $[t]$ to describe the set $\{1, \dots, t\}$. For a set $G \subseteq \mathbb{R}^n$, we use $\textup{conv}(G)$, $\textup{extr}(G)$ to denote the convex hull of $G$ and the set of extreme points of $G$ respectively.

In this paper, we generalize one of the main result in \cite{DeySantana2018}. Specifically, we show that the convex hull of a \textit{general} quadratic equation intersected with \textit{any} bounded polyhedron is \socpr. Moreover the proof is constructive, therefore adding to the literature on explicit convexification in the context of QCQPs. The formal result is as following:
\begin{theorem}\label{thm:1}
Let
\begin{eqnarray}
S: =\{ x \in \mathbb{R}^n\,|\, \ x^{\top}Qx + \alpha^{\top} x = g, \ x \in P\},\label{setS}
\end{eqnarray}
where $Q\in\mathbb{R}^{n\times n}$ is a symmetric matrix, $\alpha \in \mathbb{R}^n$, $g\in\mathbb{R}$ and $P := \{x \,|\, Ax \leq b\}$ is a polytope.  Then $\textup{conv}(S)$ is \socpr. 
\end{theorem}
Notice that we make no assumption regarding the structure or coefficients of the quadratic equation defining $S$. We require $P$ to be a bounded polyhedron, which is not very restrictive given that in global optimization the variables are often assumed to be bounded to use branch-and-bound algorithms. 

The result presented in Theorem~\ref{thm:1} is somewhat unexpected since the sum-of-squares approach would build a sequence of SDP relaxations for (\ref{setS}) in order to optimize (exactly) a linear function over $S$, while even the SDP cone of thre-by-three dimensional matrices is not \socpr \ ~\cite{fawzi2018representing}. Note that optimizing a linear function over $S$ is NP-hard, therefore, while the convex hull is \socpr, the construction involves the introduction of an exponential number of variables. 

Surprisingly, the proof of Theorem~\ref{thm:1} is fairly straightforward and it introduces a technique (new, to the best of our knowledge) to compute convex hull of certain surfaces over a compact set. In the case of Theorem~\ref{thm:1}, the key observation is that the surface defined by the quadratic equation either:
\begin{enumerate}
\item is defined as the union of two convex surfaces (see Figure~\ref{fig: two convex pieces}); or
\item it has the property that, through every point of the surface, there exists a \textit{straight} line that is \textit{entirely} contained in the surface (see Figure~\ref{fig: straight line}).
\end{enumerate}

\begin{figure}[h]
    \centering
    \begin{minipage}{0.48\textwidth}
        \centering
		\includegraphics[scale=0.32]{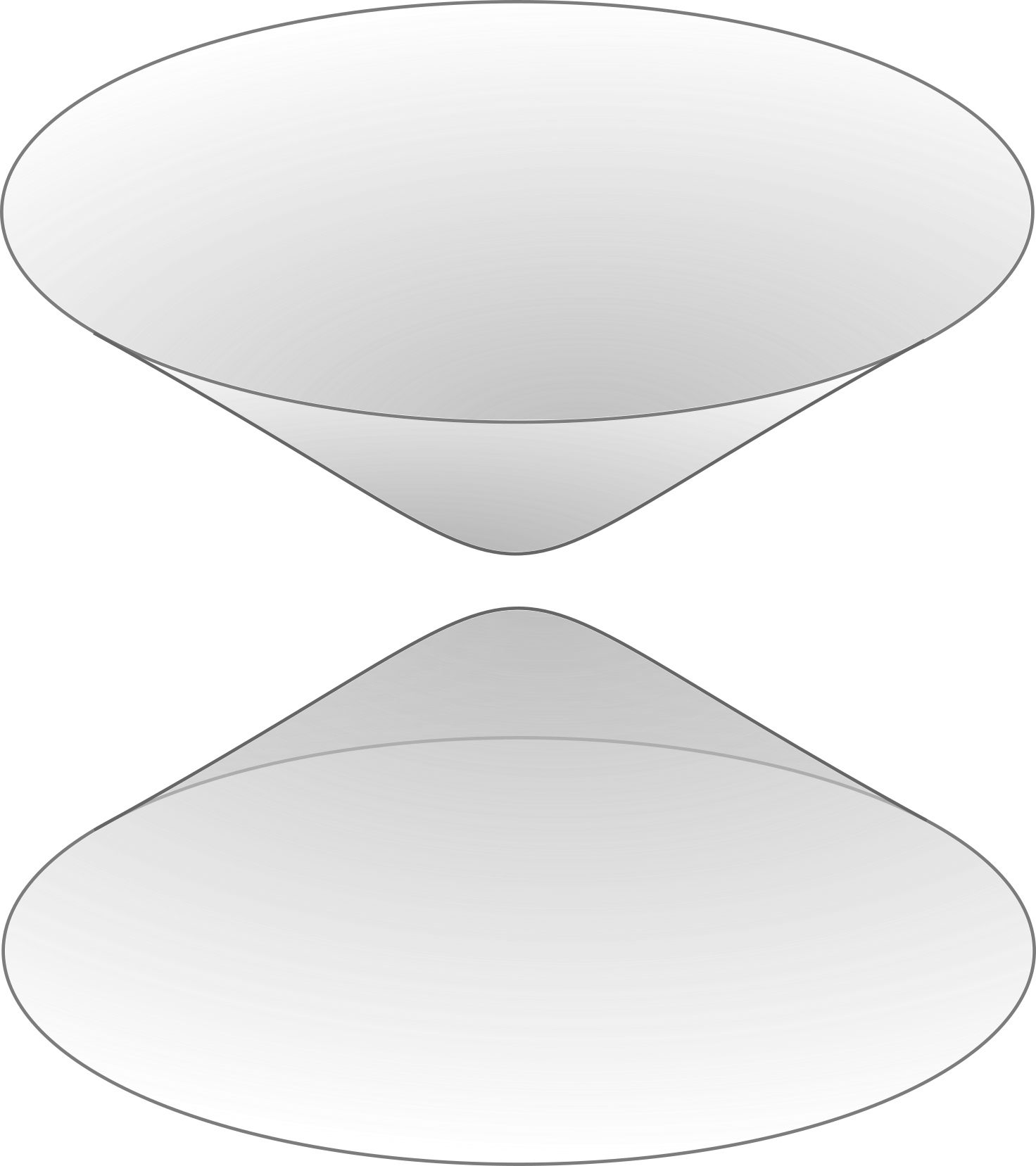}
        \caption{Two-sheets hyperboloid. The surface is the union of two convex peices.}
        \label{fig: two convex pieces}
    \end{minipage}%
    \hspace{0.3cm}
    \begin{minipage}{0.48\textwidth}
        \centering
        \includegraphics[scale=0.355]{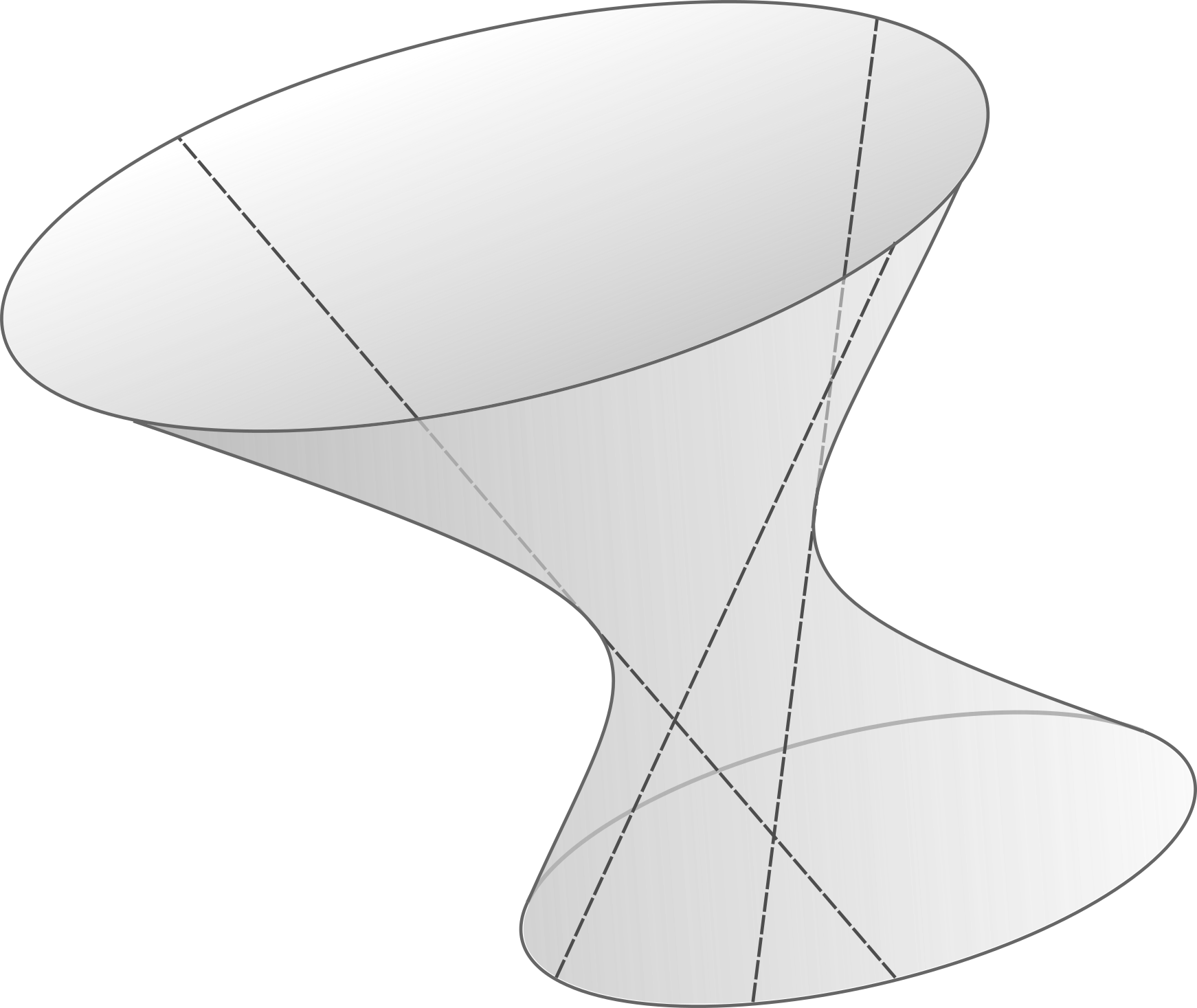} 
        \caption{One-sheet hyperboloid. Through every point of the surface, there exists a straight line that is entirely contained in the surface.}
        \label{fig: straight line}
    \end{minipage}
\end{figure}

In Case~1, we can easily obtain that the convex hull of $S$ is \socpr \ as we show in Section~\ref{section: proof of thm1}. In Case~2, no point in the interior of the polytope can be an extreme point of $S$. Observing that the convex hull of a compact set is also the convex hull of its extreme points, we intersect the surface with each facet of the polytope which will contain all the extreme points of $S$. Now, each such intersection leads to new sets with the same form as $S$ but in one dimension lower. The argument then goes by recursion. The details of the proof are presented in Section~\ref{sec:thm1}.

After we had proved Theorem~\ref{thm:1}, we learned that the property described in Case~2 is known as ``ruled surfaces'' and it has been extensively studied from both algebraic and geometric perspectives \cite{edge2011theory}. To the best of our knowledge, however, no one from the global optimization community has ever exploited such results for convexification.

\section{Proof of Theorem~\ref{thm:1}}\label{sec:thm1}
\subsection{Convex hulls via disjunctions}\label{section: convex hulls via disjunctions}

In this section, we describe a simple procedure to obtain the convex hull of a compact set $S$ using a disjunctive argument. We use this procedure to prove Theorem~\ref{thm:1} in Section~\ref{section: proof of thm1}.
Let $S$ be a compact set and let $\textup{extr}(S)$ be the set of extreme points of $S$. First, we partition the extreme points of $S$. Specifically, suppose there exist $B^1, \dots, B^k \subseteq S$ such that:
\begin{eqnarray}
S \supseteq \bigcup_{i = 1}^k B^i \supseteq \textup{extr}(S). \label{eq:A}
\end{eqnarray}
We observe that (\ref{eq:A}) implies that 
\begin{eqnarray}
\textup{conv}\left(S\right) \supseteq \textup{conv}\left(\bigcup_{i = 1}^k B^i \right) \supseteq \textup{conv}\left(\textup{extr}(S)\right) = \textup{conv}\left(S\right),
\end{eqnarray}
where the last equality holds due to $S$ being compact. Finally, we obtain that 
\begin{eqnarray}
\textup{conv}\left(S\right) = \textup{conv}\left(\bigcup_{i = 1}^k B^i \right) = \textup{conv}\left(\bigcup_{i = 1}^k \textup{conv}\left(B^i\right) \right). \label{eq:disj}
\end{eqnarray}

\begin{observation}\label{disjunctionSOC}
If $\textup{conv}(B^i)$ is \socpr \ for all $i \in [k]$, then the set 
$$\textup{conv}\left(\bigcup_{i = 1}^k \textup{conv}\left(B^i\right) \right),$$
is \socpr~\cite{ben2001lectures}.  Thus, we obtain from (\ref{eq:disj}) that $\textup{conv}(S)$ is \socpr. In addition, we obtain a constructive procedure to compute $\textup{conv}(S)$.
\end{observation}

\subsection{Reduction}\label{sec:red}
In this section, we discuss how we can apply some transformations to the set $S$ defined in (\ref{setS}) so as to re-write it in a ``canonical'' form where all the quadratic terms are squared terms. This will allows us to easily classify $S$ into Case~1 and 2 as discussed in Section~\ref{section: our result}. We start with the following observation.

\begin{observation}\label{obs:convbi}
Let $S\subseteq \mathbb{R}^n$ and let $F:\mathbb{R}^n \rightarrow \mathbb{R}^n$ be an affine map. Then 
$$\textup{conv}(F(S)) = F(\textup{conv}(S)),$$
where $F(S):= \{ Fx\,|\, x \in S\}$. 
Furthermore if $\textup{conv}(S)$ is \socpr, then $\textup{conv}(F(S)) $ is also \socpr. 
\end{observation}
Let $S$ be the set defined in (\ref{setS}). Suppose, without loss of generality, that $Q$ is a symmetric matrix. By the spectral theorem $Q= V^{\top}\Sigma V $, where $\Sigma$ is a diagonal matrix and the columns of $V$ are a set of orthogonal vectors. 
%Let $u_i = \sqrt{|\Sigma_i|}$ and let $H = u^{\top}V$. Then note that 
Letting 
$w = Vx,$
we have that 
$$S:= V^{-1}\left(\{ w\,|\, w^{\top}\Sigma w + \alpha^{\top}V^{-1}w= d, \ w \in \tilde{P}\}\right),$$
where $\tilde{P} := \{w  \,|\, AV^{-1}w \leq b\}$. 

Therefore, by Observation~\ref{obs:convbi}, it is sufficient to study the convex hull of a set of the form:
\begin{eqnarray*}
S: =\left\{ (x,y, z) \in \mathbb{R}^n\,|\, \ \sum_{i = 1}^{n_q}a_ix^2_i + \sum_{i =1}^{n_q} \alpha_i x_i +\sum_{j =1}^{n_l}\beta_jy_j = g, \ (x, y, z) \in P \right\},
\end{eqnarray*}
where $z \in \mathbb{R}^{n_o}$ does not appear in the quadratic constraints, $n_q + n_l + n_o  = n$, $a_i \neq 0$ for $i \in [n_q]$ (i.e., the rank of $Q$ is $n_q$) and $\beta_j \neq 0 $ for $j \in [n_l]$.  By completing squares, we may further write $S$ as:
\begin{eqnarray*}
S: =\{ (x,y, z) \in \mathbb{R}^n\,|\, \ \sum_{i = 1}^{n_q}\sigma(a_i) \left(\sqrt{|a_i|}x_i + \sigma(a_i) \frac{\alpha_i}{2\sqrt{|a_i|}} \right)^2+ \sum_{i =1}^{n_l}\beta_iy_i = g + \sum_{i = 1}^{n_q}  \frac{\alpha_i^2}{4a_i}, \ (x, y, z) \in P\},
\end{eqnarray*}
where $\sigma(a)$ denotes the sign of $a$. Now, since $u_i = \left(\sqrt{|a_i|}x_i + \sigma(a_i) \frac{\alpha_i}{2\sqrt{|a_i|}} \right)$  for $i \in [n_q]$ and $v_i = \beta_i y_i$ for $i \in [n_l]$ define linear bijections, it follows from Observation~\ref{obs:convbi} that it is sufficient to study the convex hull of the following set:
\begin{eqnarray}\label{eq:final}
S: =\{ (w, x,y, z) \in \mathbb{R}^{n_{q+}} \times \mathbb{R}^{n_{q-}} \times \mathbb{R}^{n_{l}}\times \mathbb{R}^{n_{o}} \,|\, \ \sum_{i = 1}^{n_{q+}}w^2_i - \sum_{j =1}^{n_{q-}}x^2_j + \sum_{k = 1}^{n_l}y_k = g, \ (w, x, y, z) \in P\}, \label{S:canonical}
\end{eqnarray}
where we may further assume that $g \geq 0$, since otherwise we may multiply the equation by $-1$ and apply suitable affine transformations to bring it back to the form of (\ref{eq:final}).

\subsection{Recursive argument to prove Theorem~\ref{thm:1}}\label{section: proof of thm1}

We begin by stating a variant of Observation~\ref{obs:convbi} that we will use twice along the proof.
\begin{lemma}\label{lemma: convex of intersection with affine set}
Let $G=\{(x,w)\in\mathbb{R}^{n_1} \times \mathbb{R}^{n_2}\,|\,  x\in G_0, \ w=C^{\top}x+h\}$, where $G_0\subseteq\rr^{n_1}$ is bounded, and $C^{\top}x + h$ is an affine function of $x$. Then,
$$\convex(G) = \{(x,w)\in\mathbb{R}^{n_1} \times \mathbb{R}^{n_2}\,|\, x\in \convex(G_0), \ w=C^{\top}x+h\}.$$
\end{lemma}
\begin{proof}
See Lemma~4 in \cite{DeySantana2018}.
\end{proof}

\subsubsection{Dealing with low dimensional polytope}\label{sec:lowdim}
Let $S$ and $P$ be defined as in (\ref{setS}).
Next, we show that we may assume without loss of generality that $P$ is full dimension. In fact, if $P$ is not full dimensional, then $P$ is contained in a non-trivial affine subspace defined by a system of linear equations $M x = f$. Without loss of generality, we may assume that $M$ has full row-rank $k$, $1\leq k<n$. Let 
$
M = 
\begin{bmatrix}
M_B & M_N
\end{bmatrix}
$ where $M_B$ is invertible. Then, we may write this system as $x_B = -M^{-1}_BM_Nx_N+M^{-1}_Bf$, where $x_B \in \mathbb{R}^{k}, \ x_N \in \mathbb{R}^{n-k}$ and, for simplicity, we assume that $x_B$ (resp. $x_N$) correspond to the first $k$ (resp. last $n-k$) components of $x$. By defining $C=-M^{-1}_BM_N$ and $h = M^{-1}_Bf$ to simplify notation, we obtain
\begin{equation}
x_B = Cx_N+h.\label{xB}
\end{equation} 
By partitioning $Q$ in sub-matrices of appropriate sizes, we  may explicitly write the quadratic equation defining $S$ in terms of $x_B$ and $x_N$ as follows:
\begin{align}
\begin{bmatrix}
x_B^{\top} & x_N^{\top}
\end{bmatrix}
\begin{bmatrix}
Q_{BB} & Q_{BN}\\
Q_{NB} & Q_{NN}
\end{bmatrix}
\begin{bmatrix}
x_B\\ 
x_N
\end{bmatrix}
+ \alpha^{\top} 
\begin{bmatrix}
x_B\\ 
x_N
\end{bmatrix} = g.\label{eq:xBxN}
\end{align}
Using (\ref{xB}), we replace $x_B$ in (\ref{eq:xBxN}) to obtain
\begin{align*}
x_N^{\top}\tilde{Q}x_N + \tilde{\alpha}^{\top} x_N = \tilde{g},
\end{align*}
where 
\begin{align*}
\tilde{Q} &= C^{\top}Q_{BB}C + C^{\top}Q_{BN} + Q_{NB}C + Q_{NN},\\\tilde{\alpha} &= 2C^{\top}Q_{BB}h + Q_{BN}^{\top}h + Q_{NB}h + C^{\top}\alpha_B + \alpha_N,\\
\tilde{g} &= g - h^{\top}Q_{BB}h - \alpha_B^{\top}h.
\end{align*}
Therefore, we may write $S$ as
\begin{align}
S: =\{ (x_B,x_N) \in \mathbb{R}^n\,|\, \ x_N^{\top}\tilde{Q}x_N + \tilde{\alpha}^{\top} x_N = \tilde{g}, \ x_N \in \tilde{P}, \ x_B = Cx_N+h\},\label{setS-full-P}
\end{align}
where $\tilde{P}$ is now a full dimensional polytope. Therefore, by Lemma~\ref{lemma: convex of intersection with affine set}, we may assume from now on that $P$ is full dimensional.

\subsubsection{Case 2: Sufficient conditions for points to not be extreme}

%Next we prove a sequence of lemmas showing that, depending on the values of $n_1,n_2$ and $n_3$ in (\ref{eq:final}), $S$ fall either in Case~1 or 2 as discussed in Section~\ref{section: our result}.
Consider the set $S$ as defined in (\ref{S:canonical}).
\begin{lemma}\label{lemma:no}
Suppose $n_o\geq 1$. If $(a, b, c, d) \in S \cap ( \mathbb{R}^{n_{q+}} \times  \mathbb{R}^{n_{q-}} \times \mathbb{R}^{n_{l}} \times \mathbb{R}^{n_{o}})$ where $(a, b, c, d) \in \textup{int}(P)$, then $(a, b, c, d)$ is not an extreme point of $S$.
\end{lemma} 
\begin{proof}
Since $(a, b, c, d) \in \textup{int}(P)$, there exists a vector $\delta \in \mathbb{R}^{n_o}\setminus \{0\}$ such that $(a,b,c,d + \delta), (a,b,c,d - \delta) \in P$. Clearly these points are in $S$ as well and, therefore, $(a, b, c, d)$ is not an extreme point of $S$
\end{proof}
	
\begin{lemma}\label{lemma:n3 geq 2}
Suppose $n_o  = 0$ and $n_l \geq 2$. If $(a, b, c) \in S \cap ( \mathbb{R}^{n_{q+}} \times  \mathbb{R}^{n_{q-}} \times \mathbb{R}^{n_{l}}) $ where $(a, b, c) \in \textup{int}(P)$, then $(a, b, c)$ is not an extreme point of $S$.
\end{lemma}
\begin{proof}
Since $n_l \geq 2$, $(a,b,c_1 \pm \lambda, c_2 \mp \lambda, \dots, c_{n_3})$ are feasible for sufficiently small positive values of $\lambda$. Therefore, $(a, b, c)$ is not an extreme point.
\end{proof}

\begin{lemma}\label{lemma:n3 eq 1 and n1 n2 geq 1}
Suppose $n_o = 0$, $n_{q+}, n_{q-}\geq 1$ and $n_l =1$. If $(a, b, c) \in S \cap ( \mathbb{R}^{n_{q+}} \times  \mathbb{R}^{n_{q-}} \times \mathbb{R}^{n_{l}})$ where $(a, b, c) \in \textup{int}(P)$, then $(a, b, c)$ is not an extreme point of $S$.
\end{lemma}
\begin{proof}
Since $n_{q+}, n_{q-} \geq 1$, and $n_l =1$, $(a_1 + \lambda, a_2, \dots, a_{n_{q+}}, b_1 + \lambda, b_2, \dots, b_{n_{q-}}, c + 2\lambda(-a_1 + b_1)$ are feasible for sufficiently small positive and negative values of $\lambda$. Therefore, $(a, b, c)$ is not an extreme point.
\end{proof}

\begin{lemma}\label{lemma: n1 geq 2 and n2 geq 2} 
Suppose $n_o = 0$, $n_{q+} \geq 2$, $n_{q-} \geq 1$ and $n_l=0$. If $(a, b) \in S \cap ( \mathbb{R}^{n_{q+}} \times  \mathbb{R}^{n_{q-}} )$ where $(a, b) \in \textup{int}(P)$, then $(a, b)$ is not an extreme point of $S$.
\end{lemma}
\begin{proof} We show that there exists a straight line through $(a,b)$ that is entirely contained in the surface defined by the quadratic equation. More specifically, we prove that there exists a vector $(u,v) \in (\mathbb{R}^{n_{q+}} \times \mathbb{R}^{n_{q-}})\setminus \{0\}$ such that the line $\{(a,b) + \lambda (u,v)\,|\, \lambda \in \mathbb{R}\}$ satisfies the quadratic equation and therefore, $(a, b)$ being in the interior of $P$ cannot be an extreme point of $S$.  
We consider two cases:
\begin{enumerate}
\item $(a,b) \neq \textbf{0}$: Then observe that $a \neq \textbf{0}$, since otherwise we would have $a = \textbf{0}$ and $b = \textbf{0}$, because $g\geq 0$. 
%For simplicity, assume $1 \in [n_2]$. 
Observe that 
\begin{eqnarray}
\sum_{i =1}^{n_{q+}}a_i^2 = g + \sum_{j = 1}^{n_{q-}} b_j^2 \geq b_1^2 \Leftrightarrow \frac{|b_1|}{\|a\|_2} \leq 1.\label{eq:dist}
\end{eqnarray}
Next, observe that:
\begin{eqnarray}
g &=& \sum_{i = 1}^{n_{q+}}( a_i + \lambda u_i)^2 -\sum_{i = 1}^{n_{q-}}(b_i + \lambda v_i)^2 \ \forall \lambda \in \mathbb{R} \nonumber\\
\Leftrightarrow  g & = & \left( \sum_{i = 1}^{n_{q+}}a_i^2 -  \sum_{i = 1}^{n_{q-}}b_i^2 \right) + \lambda^2 \left( \sum_{i = 1}^{n_{q+}}u_i^2 -  \sum_{i = 1}^{n_{q-}}v_i^2 \right) + 2\lambda \left(\sum_{i = 1}^{n_{q+}}a_iu_i -  \sum_{i = 1}^{n_{q-}}b_iv_i \right) \ \forall \lambda \in \mathbb{R}\nonumber \\
\Leftrightarrow  0 & = & \lambda \left(  \sum_{i = 1}^{n_{q+}}u_i^2 -  \sum_{i = 1}^{n_{q-}}v_i^2 \right) + 2\left(\sum_{i = 1}^{n_{q+}}a_iu_i -  \sum_{i = 1}^{n_{q-}}b_iv_i \right) \ \forall \lambda \in \mathbb{R} \nonumber\\
\Leftrightarrow &&   \sum_{i = 1}^{n_{q+}}u_i^2 -  \sum_{i = 1}^{n_{q-}}v_i^2 = 0, \ \sum_{i = 1}^{n_{q+}}a_iu_i -  \sum_{i = 1}^{n_{q-}}b_iv_i =0. \label{eq:1}
\end{eqnarray}
Suppose we set $v_1 = 1$ and $v_j =0$ for all $j \in \{2,\dots,n_{q-}\}$. Then satisfying (\ref{eq:1}) is equivalent to finding real values of $u$ satisfying:
\begin{eqnarray*}
\sum_{i = 1}^{n_{q+}}u_i^2  =1, \ \
\sum_{i = 1}^{n_{q+}}a_iu_i  = b_1.
\end{eqnarray*}
This is the intersection of a circle of radius $1$ in dimension two or higher (since $n_{q+}\geq 2$ in this case) and a hyperplane whose distance from the orgin is $\frac{|b_1|}{\|a\|_2}$. Since, by (\ref{eq:dist}), we have that this distance is at most $1$, the hyperplane intersects the circle and therefore we know that a real solution exists. 
\item $(a,b) = \textbf{0}$: In this case, observe that $g = 0$ and then $\textbf{0}$ is a convex combination of
$$( \underbrace{\pm\lambda, 0, \dots, 0}_{\textup{first } n_{q+} \textup{ components} }, \underbrace{\pm\lambda, 0, \dots, 0}_{\textup{second } n_{q-} \textup{ components}})$$  for sufficiently small $\lambda > 0$.
\end{enumerate}
\end{proof}

\subsubsection{Case 1: Sufficient conditions for convex hull to be \socpr}

In this section, we repeatedly use the following result from~\cite{Tawarmalani2013}.
\begin{theorem}\label{thm:2}
Let $G \subseteq \mathbb{R}^n$ be a convex set and let $f: \mathbb{R}^n \rightarrow \mathbb{R}$ be a continuous function.  Then 
\begin{eqnarray}
\textup{conv}\left(\{G \cap \{x\,|\, f(x) = 0\}\} \right) = \textup{conv}\left(\{G \cap \{x\,|\, f(x) \leq 0\}\} \right) \cap \textup{conv}\left(\{G \cap \{x\,|\, f(x) \geq 0\}\} \right). \nonumber
\end{eqnarray}
\end{theorem}

For the two lemmas that follows, consider the notation of $S$ defined in (\ref{S:canonical}).
\begin{lemma} \label{lemma:n3 leq 1 and n1 n2 eq 0}
Suppose $n_o = 0$, $n_l \leq 1$. If $n_{q+} = 0$ or $n_{q-} = 0$, then $\textup{conv}(S)$ is \socpr.
\end{lemma}
\begin{proof}% Let $(a, b, c)  \in S \cap ( \mathbb{R}^{n_{q+}} \times  \mathbb{R}^{n_{q-}} \times \mathbb{R}^{n_{l}})$. 
%Let $z = c_1$ if $n_l=1$ and $z = 0$ if $n_l = 0$. 
We consider two cases.
\begin{enumerate}
\item $n_{q-} = 0$: Let $(w, y)  \in S \cap ( \mathbb{R}^{n_{q+}} \times \mathbb{R}^{n_{l}})$. Let $y = y_1$ if $n_l=1$ and $y = 0$ if $n_l = 0$. In this case, $g - y$ is non-negative for all feasible values of $y$ and we can use the identity $t = \frac{(t+1)^2 - (t-1)^2}{4}$ to write $S = S' \cap S''$, where:
\begin{eqnarray*}
S':= \{(w,y) \in P\,|\, \|2w_1, \dots, 2w_{n_{q+}}, (g - y - 1)\| \leq (g - y + 1)\},\\
S'':= \{(w,y) \in P \,|\, \|2w_1, \dots, 2w_{n_{q+}}, (g - y - 1)\| \geq (g - y + 1)\}.
\end{eqnarray*}
Notice that $S'$ is a \socpr \ convex set. Also notice that $S''$ is a reverse convex set intersected with a polytope and hence $\textup{conv}(S''\cap P)$ is polyhedral and contained in $P$ (see \cite{Hillestad1980},Theorem~1). Therefore, by Theorem~\ref{thm:2}, we have that $\convex{(S)} = \convex{(S')} \cap \convex{(S'')}$ is \socpr.
\item $n_{q+} = 0$: Let $(x, y)  \in S \cap ( \mathbb{R}^{n_{q+}} \times \mathbb{R}^{n_{l}})$. Let $y = y_1$ if $n_l=1$ and $y = 0$ if $n_l = 0$. In this case, $g - y$ is non-positive for all feasible values of $y$ and may write $S = S' \cap S''$, where:
\begin{eqnarray*}
S':= \{(x,y) \in P \,|\, \|2x_1, \dots, 2x_{n_{q-}}, (y - g - 1)\| \leq (y - g + 1)\},\\
S'':= \{(x,y) \in P \,|\, \|2x_1, \dots, 2x_{n_{q-}}, (y - g - 1)\| \geq (y - g + 1).\}
\end{eqnarray*}
Therefore, as in the previous case,
$\convex{(S)}$ is \socpr.
\end{enumerate}
\end{proof}

\begin{lemma}\label{lemma: n1 leq 1 and n3 eq zero}
Suppose $n_{q+} \leq 1$ and $n_l = n_o =0$. Then $\textup{conv}(S)$ is \socpr. 
\end{lemma}
\begin{proof}
If $n_{q+}=0$, then $S$ is empty set or contains a single point, the origin. 

Therefore, consider the case where $n_{q+} = 1$, thus $w=w_1$. 
%and let $x=x_1$ to simplify notation. 
Notice that $S = S' \cap S''$, where
\begin{align*}
S':= \{(w,x) \in \mathbb{R}^{1} \times \mathbb{R}^{n_{q-}} \,|\, \ w^2 \geq g+\sum_{j=1}^{n_{q-}}x_j^2, \  (w,x) \in P\},\\
S'':= \{(w,x) \in \mathbb{R}^{1} \times \mathbb{R}^{n_{q-}}\,|\, \ w^2 \leq g+\sum_{j=1}^{n_{q-}}x_j^2, \ (w,x) \in P\}.
\end{align*}
By Theorem~\ref{thm:2}, $\textup{conv}(S) = \textup{conv}(S') \cap \textup{conv}(S'')$. Next, we show that both $\textup{conv}(S')$ and $\textup{conv}(S'')$ are \socpr.
Notice that $S'$ is the union of the following two \socpr \ sets:
\begin{align*}
S'_+&:= \left\{ (w,x) \in \mathbb{R}^{1} \times \mathbb{R}^{n_{q-}}\,|\, \ w \geq \left(g+\sum_{j=1}^{n_{q-}}x_j^2\right)^{\frac{1}{2}}, \ w \geq 0, \ (w,x) \in P\right\},\\
&=\text{Proj}_{w,x}\left(\left\{(w,x,t) \in \mathbb{R}^{1} \times \mathbb{R}^{n_{q-}} \times \mathbb{R}\,|\, \ w \geq \left((\sqrt{g}t)^2+\sum_{j=1}^{n_{q-}}x_j^2\right)^{\frac{1}{2}}, \ x \geq 0, \ t = 1, \ (w,x) \in P \right\}\right),\\
S'_-&:= \left\{(w,x) \in \mathbb{R}^{1} \times \mathbb{R}^{n_{q-}}\,|\, \ -w \geq \left(g+\sum_{j=1}^{n_{q-}}x_j^2\right)^{\frac{1}{2}}, \ w \leq 0, \ (w,x) \in P\right\}\\
&= \text{Proj}_{w,x}\left(\left\{(w,x,t) \in \mathbb{R}^{1} \times \mathbb{R}^{n_{q-}}\times\mathbb{R}\,|\, \ -w \geq \left(\sqrt{g}t)^2+\sum_{j=1}^{n_{q-}}x_j^2\right)^{\frac{1}{2}}, \ w \leq 0, \ t=1, \ (w,x) \in P\right\}\right).
\end{align*}
Thus, $\textup{conv}(S') = \textup{conv}(S'_+ \cup S'_-)$ is \socpr. 

Notice that $S'' = \{(w,x) \in \mathbb{R}^{1} \times \mathbb{R}^{n_{q-}}\,|\, \ |w| \leq (g+\sum_{j=1}^{n_{q-}}x_j^2)^{\frac{1}{2}}, \ (w,x) \in P\}$ and is therefore the union of two sets:
\begin{align*}
S''_+&:= \left\{ (w,x) \in \mathbb{R}^{1} \times \mathbb{R}^{n_{q-}}\,|\, \ w \leq\left(g+\sum_{j=1}^{n_{q-}}x_j^2\right)^{\frac{1}{2}}, \ w \geq 0, \ (w,x) \in P\right\},\\
S''_-&:= \left\{(w,x) \in \mathbb{R}^{1} \times \mathbb{R}^{n_{q-}}\,|\, \ -w \leq \left(g+\sum_{j=1}^{n_{q-}}x_j^2\right)^{\frac{1}{2}}, \ w \leq 0, \ (w,x) \in P\right\},
\end{align*}
each of them being a reverse convex set intersected with a polyhedron. Therefore, $\textup{conv}(S''_+)$ and $\textup{conv}(S''_-)$ are polyhedral and therefore
$\textup{conv}(S'') =\textup{conv}(\textup{conv}(S''_+) \cup \textup{conv}(S''_-))$ is a polyhedral set.
\end{proof}
\subsubsection{Proof of Theorem~\ref{thm:1}}

Finally, we bring the pieces together to prove Theorem~\ref{thm:1}. 
\begin{proof}(of Theorem~\ref{thm:1})
Let $S(n)$ be defined as in (\ref{eq:final}), where $n=n_{q+} + n_{q-} + n_l + n_o$ is the dimension of the space in which $S$ is defined and without loss of generality $P$ is full-dimensional (Section~\ref{sec:lowdim}). The proof goes by induction on $n$.  Notice that $S(1)$ is a polytope and hence $\textup{conv}(S(1))$ is \socpr. Suppose $S(n)$ is \socpr. We show that $S(n+1)$ is \socpr \ as well. If $n_o = 0$, $n_l\leq 1$, and $n_{q+} =0$ or $n_{q-}=0$, then the result follows from Lemma~\ref{lemma:n3 leq 1 and n1 n2 eq 0}. Similarly, if $n_o = 0$, $n_{q+} \leq 1$ and $n_{l}=0$, then the result follows from Lemma~\ref{lemma: n1 leq 1 and n3 eq zero}. Otherwise, it follows from Lemma~\ref{lemma:no}, \ref{lemma:n3 geq 2}, \ref{lemma:n3 eq 1 and n1 n2 geq 1} and \ref{lemma: n1 geq 2 and n2 geq 2} that no point in the interior of $P$ can be an extreme point of $S(n+1)$. Let $N$ be the number of facets of $P$, each of which given by one equation of the linear system $Fx = f$. Let $B^i = S(n+1) \cap \{x\in\mathbb{R}^{n+1}\,|\, F_{i.}x = f_i\}$ be the intersection of $S(n+1)$ with the $i$th facet of $P$. By the discussion in Section~\ref{section: convex hulls via disjunctions}, it is enough to show that the convex hull of each $B^i$ is \socpr. Let $i\in\{1,\dots,N\}$. Choose $j_0$ such that $F_{ij_0}\neq 0$. For simplicity, suppose $j_0 = 1$. Then, we may write $B^i = \{x\in\mathbb{R}^{n+1}\,|\, (x_2,\dots,n_{n+1})\in B^i_0, \ x_1 = b_i - \sum_{j=2}^{n+1}F_{ij}x_j\}$, where $B_0^i$ is obtained from $B^i$ by replacing $x_1 = f_i - \sum_{j=2}^{n+1}F_{ij}x_j$ in all the constraints defining $S(n+1)$. Now $\textup{conv}(B_0^i) \subseteq \mathbb{R}^n$ is \socpr \ by induction hyptothesis. Therefore, $\textup{conv}(B^i)$ is \socpr \ by Lemma~\ref{lemma: convex of intersection with affine set}.
\end{proof}

\section*{Acknowledgments}

Funding: This work was supported by the NSF CMMI [grant number 1562578] and the CNPq-Brazil [grant number 248941/2013-5].

%\newpage
\bibliographystyle{plain}
\bibliography{QCQP_bibfile_v1}

\end{document}